\documentclass[preprint,10pt]{elsarticle}

\usepackage{amssymb,amsthm}
\usepackage{graphicx,amsmath,bm,color,geometry,subfig}
\usepackage{algorithmic,algorithm,diagbox} 
\usepackage{amsmath}
\usepackage{bbm}

\usepackage[colorlinks,
            linkcolor=blue,
            anchorcolor=blue,
            citecolor=blue]{hyperref}

\newtheorem{theorem}{Theorem}

\newtheorem{lemma}{Lemma}

\journal{ArXiv}

\begin{document}

\begin{frontmatter}



\title{Resolution invariant deep operator network for PDEs with complex geometries \tnoteref{label-title}}
\tnotetext[label-title]{This work is partially supported by the National Natural Science Foundation of China (NSFC) under grant number 12101407, the Chongqing Entrepreneurship and Innovation Program for Returned Overseas Scholars under grant number CX2023068, and the Fundamental Research Funds for the Central Universities under grant number 2023CDJXY-042.}

\author[label-addr1]{Jianguo Huang}
\ead{huangjg@shanghaitech.edu.cn}

\author[label-addr2,label-addr3]{Yue Qiu\corref{label-cor}}
\ead{qiuyue@cqu.edu.cn}
\cortext[label-cor]{Corresponding author.}

\affiliation[label-addr1]{
	organization={School of Information Science and Technology, ShanghaiTech University},
	city={Shanghai},
	postcode={201210},
	country={China.}}

\affiliation[label-addr2]{
	organization={College of Mathematics and Statistics, Chongqing University},
	city={Chongqing},
	postcode={401331},
	country={China.}}
	
\affiliation[label-addr3]{
	organization={Key Laboratory of Nonlinear Analysis and its Applications (Chongqing University), Ministry of Education},
	city={Chongqing},
	postcode={401331},
	country={China.}}

\begin{abstract}
Neural operators (NO) are discretization invariant deep learning methods with functional output and can approximate any continuous operator. NO have demonstrated the superiority of solving partial differential equations (PDEs) over other deep learning methods. However, the spatial domain of its input function needs to be identical with its output, which limits its applicability. For instance, the widely used Fourier neural operator (FNO) fails to approximate the operator that maps the boundary condition to the PDE solution. To address this issue, we propose a novel framework called resolution-invariant deep operator (RDO) that decouples the spatial domain of the input and output. RDO is motivated by the Deep operator network (DeepONet) and it does not require retraining the network when the input/output is changed compared with DeepONet. RDO takes functional input and its output is also functional so that it keeps the resolution invariant property of NO. It can also resolve PDEs with complex geometries whereas NO fail. Various numerical experiments demonstrate the advantage of our method over DeepONet and FNO.
\end{abstract}



\end{frontmatter}



\section{Introduction}\label{sec:Introduction}

Partial differential equations (PDEs) have many real world applications in multiphysics, biological and economic systems~\cite{di2021deeponet,qiu2020efficient,hesthaven2018non}. Scientists and engineers have been working for centuries to compute the solutions of PDEs accurately. Traditional methods, such as the finite element method~\cite{rao2017finite,zienkiewicz2005finite}, the finite difference methods~\cite{strikwerda2004finite, thomas2013numerical}, and the spectral methods~\cite{xiu2010numerical,shen2011spectral} have made numerous success. However, these methods need expensive computation resources, especially for inverse problems and hybrid problems of high dimension~\cite{butler2011posteriori, kaipio2006statistical}.

In the last decades, machine learning, including deep learning methods, have made huge success in real-life applications, such as computer vision~\cite{kingma2018glow,pratt2017fcnn}, natural language processing~\cite{vaswani2017attention,floridi2020gpt}, automatic speech recognition~\cite{yu2016automatic, ren2019almost}, and so on. It is noteworthy that PDEs promote the development of machine learning, such as the diffusion models~\cite{DBLP:conf/iclr/0011SKKEP21}, ResNet~\cite{he2016deep}, and Neural ODE~\cite{DBLP:conf/nips/ChenRBD18}. 
Simultaneously, machine learning also succeed in computational science, specifically, neural networks are harnessed to solve PDEs numerically. The first category of approaches is called function regression, where neural networks are used to approximate the solution of  PDEs~\cite{guo2022monte,gao2023failure,hu2022augmented,CHEN2021110666,YU2022114823}, such as the widely used physics informed neural networks (PINNs)~\cite{raissi2019physics}.  PINNs can solve the PDEs without labeled data by  embedding physical information, which is very different from traditional machine learning methods. PINNs can also address inverse problems more effectively compared with conventional approaches grounded in the Bayesian framework. However, PINNs are limited to approximate the solution for specific instances of a given PDE problem. This implies that if the parameters of the PDE change, the neural network needs to undergo retraining. This characteristic makes PINNs less efficient in scenarios where parametric PDEs must be repeatedly solved for varying parameters~\cite{regazzoni2021physics}. The other category is called operator regression or neural operator (NO), which was proposed for parametric PDEs. Neural operators utilize neural networks to learn the solution operator for the same kind of PDE problems rather than the solution of a specific PDE problem. For neural operators, the input could be the functions that represent the boundary conditions, initial conditions selected from a properly designed input space, or a vector representing the parameters in the parametric space of PDEs. 

DeepONet~\cite{Lu2021,DENG2022411, lanthaler2022error, zhu2023fourier}, inspired by the work in~\cite{chen1995universal}, utilizes a linear combination of finite nonlinear functions to approximate the solution operator from one Banach space to another Banach space.  \citet{kaltenbach2022semi} employed an invertible neural network~\cite{meng2023physics,guo2022normalizing} to establish a mapping between the parametric input and the weights of linear combinations. \citet{li2020fourier} proposed the Fourier neural operator (FNO) to approximate the Green function. FNO exhibits resolution invariant, i.e., the network could be trained using low resolution data and the learned network could be directly generalized to high resolution output prediction beyond the discretization of the training data. Besides the low-resolution data being used to minimize the empirical risk error during the training, high-resolution physical information could also be embedded into the loss function similar to PINNs.  This approach is called physic-informed neural operator (PINO)~\cite{li2021physics}. Benefiting from multi-resolution information, PINO has higher precision than the standard PINNs and other variants of PINNs. In FNO, a pointwise mapping is used to lift the input function to higher dimension (or channels) functions, which can extract additional information from the input function. However, this mapping disregards the interrelation among distinct positions within the domain of the input function. To capture these relationships, the attention mechanisms~\cite{vaswani2017attention, geneva2022transformers} and graph neural networks~\cite{lotzsch2022learning} have been employed. \citet{lotzsch2022learning} applied the graph neural networks to resolve different boundary value problems. \citet{kissas2022learning} employed the attention mechanism based on the known query locations to enhance the performance of the trunk net of the standard DeepONet. \citet{li2023transformer} utilized the standard transformer to approximate a unified  operator for a specific type of PDEs defined on different domains, encompassing both regular and irregular domains.  

As researchers point out that FNO can neither predict flexible location nor solve the PDEs directly defined on arbitrary domains~\cite{lu2022comprehensive}. Therefore, \citet{li2022fourier} proposed the   geometry-aware Fourier neural operator (Geo-FNO) that converts the broader physical domain to a standardized latent domain. However, Geo-FNO still struggles with handling the case that the input function and solution function are defined over different domains. DeepONet can naturally solve this issue by decoupling the input and output domains, while DeepONet fails to keep the same architecture for different input resolutions and retraining is necessary when the input/output resolution changes.

In this paper, we propose a novel neural operator called resolution-invariant deep operator (RDO) for PDEs where the input function and the solution function are defined on different domains. Our work is motivated by DeepONet and compared with DeepONet, RDO has the property of resolution invariant, i.e., once trained RDO can make predictions of varying resolutions without network retraining.  Specifically, we use a novel neural operator to replace the branch net which is a fully connected neural network (FNN) in DeepONet. This novel operator combines a neural operator such as FNO with an integral operator and yields a mapping from an infinite-dimension Banach space to a finite-dimension space. This makes RDO can handle with functional input. Benefiting from the decoupling of the input and output domain, RDO can also handle with PDEs defined on irregular domain and time-dependent PDEs. Numerical experiments demonstrate the superiority of RDO to DeepONet and FNO.

This paper is organized as follows. Section~\ref{sec:pre} presents the problem framework and introduces the baseline method, DeepONet. In Section~\ref{sec:method}, we give the details and the approximation theorem of RDO. In Section~\ref{sec:Numerical}, we compare the performance of RDO with  DeepONet and FNO using three benchmark problems. We summarize this paper in Section~\ref{sec:Conclusion}.

\section{Operator regression}\label{sec:pre}

\subsection{Problem settings}
In this section, we introduce the neural operator for PDE  problems. Considering the PDE given by
\begin{align*}
    \mathcal{L}(a(x),u(x))=f(x),\ x\in\Omega\\
    u(x) = b(x),\ x\in\partial \Omega 
\end{align*}
where $\Omega$ is a bounded domain, $a(\cdot)\in \mathcal{A} \subseteq \mathcal{V}$ is the coefficient function, $u(\cdot)\in \mathcal{U}$ is the unknown solution, and $f(\cdot)$ is the governing function. 
Here, $\mathcal{V}$ and $\mathcal{U}$ are infinite-dimension Banach spaces and $\mathcal{L}: \mathcal{V} \times \mathcal{U}\rightarrow f$ is a nonlinear or linear partial differential operator.
The function  $b\in \mathcal{B}\subseteq \mathcal{V}$ represents the boundary condition and can be considered as a parameter of the PDE.
We want to parameterize the solution operator $G:\mathcal{A}\times\mathcal{B} \times f\rightarrow \mathcal{U}$ using neural networks.
In this paper, the function $f$ is fixed, only one of $a(\cdot)$ and $b(\cdot)$ is varying. Consequently, the solution operator $G$ could be simplified to a mapping from $\mathcal{V}$ to $\mathcal{U}$.
We aim to build a parametric model $\mathcal{G}_{\theta}$ to approximate the solution operator ${G}$, where $\theta$ represents the parameters of neural networks. The loss functional denoted by $C:\mathcal{U}\times \mathcal{U}\rightarrow \mathbbm{R}$ and the minimizing problem is given by
\begin{align}\label{eq:min_a}
    \min_{\theta\in\Theta} \mathbbm{E}_{a\sim\pi}[C({G}(a),\mathcal{G}_\theta(a))],
\end{align}
where $\pi$ represents the probability distribution on $\mathcal{A}$, $C:\mathcal{U}\times \mathcal{U} \rightarrow \mathbb{R}^+ $ measures the distance between two functions, and $\Theta$ is the set of trainable parameters.
Given the training data set $\{(a^{(i)},u^{(i)})\}_{i=1}^{N}$ where the input and output pairs are indexed by $i$, Equation~\eqref{eq:min_a} can be approximated by an empirical risk minimization problem~\cite{Wang2021GeneralizingTU}, 
\begin{align}\label{eq:min_a_empirical}
    \min_{\theta\in\Theta} \sum_{i=1}^{N}[C({G}(a^{(i)}),\mathcal{G}_\theta(a^{(i)}))].
\end{align}

In this paper, we unify mathematical notations and assume that the vector-valued function $a:\mathbb{R}^{d_1}\rightarrow \mathbb{R}^{d_a}$ refers to the input function of the operator ${G}$, and the vector-valued  function $u:\mathbb{R}^{d_2}\rightarrow \mathbb{R}^{d_u}$ refers to the output of $G$. For simplicity, we assume that $d_u=1$ and $d_a=1$. Meanwhile, let $\mathcal{D}_1\subset$ $ \mathbb{R}^{d_1}, \mathcal{D}_2\subset \mathbb{R}^{d_2}$ represent the physical domain of $a$ and $u$, respectively. 

\subsection{DeepONet}
 DeepONet utilizes a linear combination of finite basis functions to approximate the infinite-dimension operator. The coefficients are approximated by the branch net whose input is the input function and the basis functions are approximated by the trunk net whose input is the predicted location. Figure~\ref{fig:deeponet_structure} depicts the architecture of DeepONet and with this network structure, DeepONet can query arbitrary location in the domain $\mathcal{D}_2$ of the solution function, thereby directly solving a variety of problems, including those involving irregular domain PDEs and time-dependent PDEs~\cite{lu2022comprehensive}. 
\begin{figure}[H]
    \begin{center}
            \includegraphics[width=0.75\linewidth]{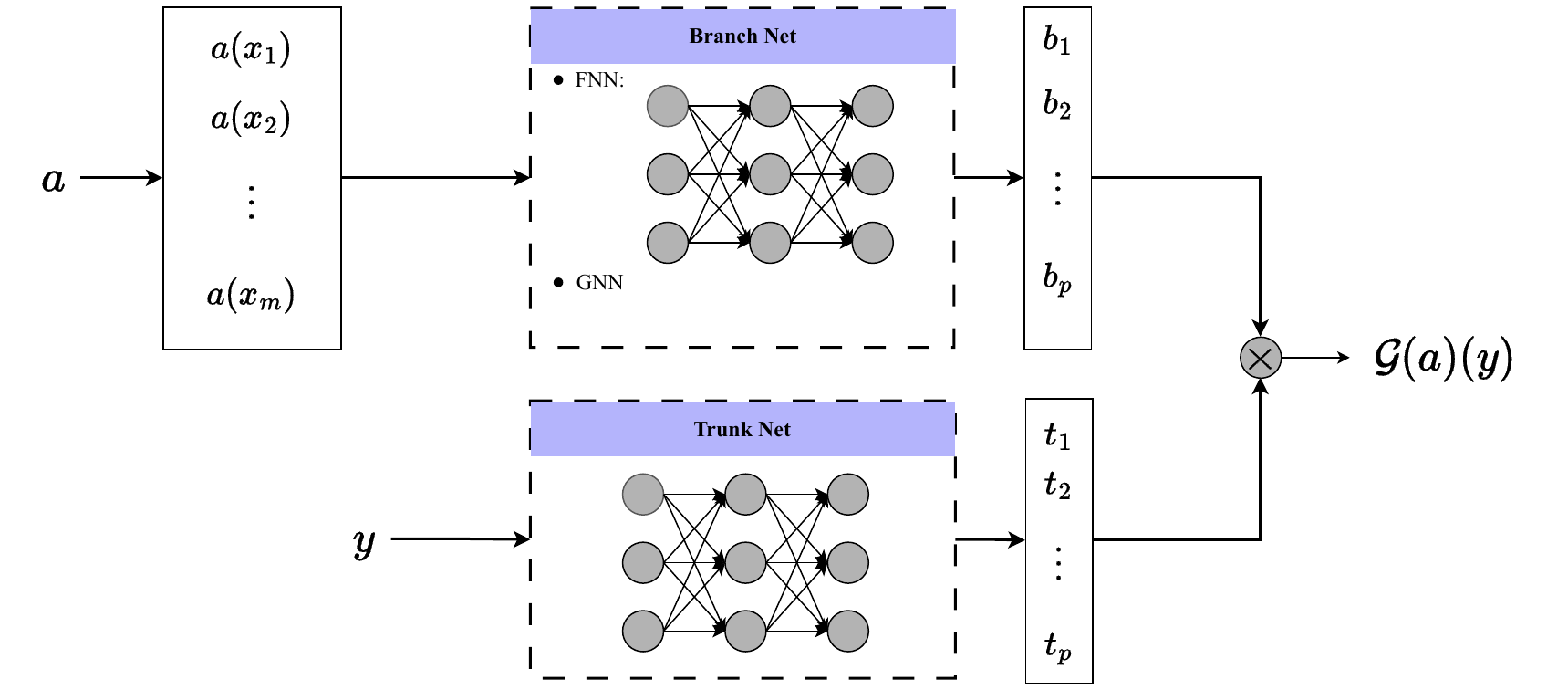}
      \end{center}
    \caption{DeepONet architecture.}\label{fig:deeponet_structure}
\end{figure}

To work with the input function $a$ numerically, DeepONet discretizes the input function $a(\cdot)$ and evaluate $a(\cdot)$ at a set of fixed locations $\{x_i\}_{i=1}^m$, which in turn gives the pointwise
evaluation $(a(x_1), a(x_2), \dots, a(x_m))$.
Let $\mathbf{b}: \mathbb{R}^{m} \rightarrow \mathbb{R}^{p}$ and $\mathbf{t}:\mathcal{D}_2\rightarrow \mathbb{R}^{p}$ represent the branch net and the trunk net, respectively, and let  $\{b_k\}_{k=0}^p\in\mathbb{R}$, and $\{t_k\}_{k=1}^p \in\mathbb{R}$ the corresponding output, respectively.
The output of DeepONet is expressed by 
\[
\mathcal{G}_{\theta}(a)(y) = \mathbf{b}(a) \mathbf{t}(y)^T  + b_0 = \sum_{k=1}^{p} b_k(a) t_k(y) + b_0,
\]
where $ y\in \mathcal{D}_2$ and $b_0\in\mathbb{R}$ is a bias.
For the framework of DeepONet illustrated in Figure~\ref{fig:deeponet_structure}, the fully connected neural network (FNN) is commonly used as the trunk net while the network structure of the branch net depends on specific applications. For example, when the discretization grids of $a$ is unstructured,  the graph neural network (GNN)~\cite{lotzsch2022learning,jiang2023phygnnet,Horie2022PhysicsEmbeddedNN} as the branch net is preferred. 

\section{Proposed Architecture}\label{sec:method}
In this section, we introduce a novel framework called resolution-invariant deep operator (RDO) based on DeepONet and provide the corresponding universal approximation theorem. Subsequently, we give a comprehensive overview of the parameterisation of RDO, including the parameterisation of subnetworks.

\subsection{RDO}\label{sec:RDO}
While DeepONet is a flexible framework, it lacks the ability to retain a consistent neural network structure for input functions of differing resolutions. The primary reason is that the branch net approximates a mapping from a finite-dimension linear space to another finite-dimension space. 
Specifically, the input of the branch net is a discretization of the input function rather than the function itself. This motivates us to design a novel neural network structure that constructs a mapping from an infinite Banach space to a finite-dimension space to replace the branch net from DeepONet. 

We represent the proposed network for the branch net by $\mathcal{G}_1: \mathcal{A}\rightarrow \mathbb{R}^p$, and we construct $\mathcal{G}_1$ by $\mathcal{G}_1 := \mathcal{T}\circ \mathcal{G}_0$. Here $\mathcal{T}$ refers to the integral operator, and $\mathcal{G}_0$ is a commonly used neural operator in~\cite{DBLP:journals/jmlr/KovachkiLLABSA23}. Denote $\phi_T(x):= \mathcal{G}_0(a)(x)\in \mathbb{R}^p$, and let $ \phi_{T,i}(x)$ represent the $i$-th element of the vector-valued mapping $\phi_{T}(x)$. The integral transformation $\mathcal{T}$ is given by 
\begin{equation}\label{eq:integral_equation}
    \mathcal{T}(\phi_T) = \int_{D_1} \phi_T(x) dx =  \begin{bmatrix}
        \int_{D_1} \phi_{T,1}(x) dx \\
        \int_{D_1} \phi_{T,2}(x) dx \\
        \vdots\\
        \int_{D_1} \phi_{T,p}(x) dx
        \end{bmatrix}
        =  \begin{bmatrix}
            b_1\\
            b_2\\
            \vdots\\
            b_p
            \end{bmatrix}   , 
\end{equation}
which maps the function $\phi_T(x)$ to a real vector $(b_1,b_2,\dots,b_p)^T$. Then, the overall framework of RDO can be defined by 
\[
\mathcal{G}(a)(y)=\mathbf{f}(y) *\mathcal{G}_1(a),
\]
where $*$ represents the inner product and  $\mathbf{f}(\cdot): \mathcal{D}_2 \rightarrow \mathbb{R}^{ p}$ is another parametric model. Subsequently,  the general computation flow of RDO is given by 
$$a(x) \underbrace{\xrightarrow[]{\mathcal{G}_0} \phi_T(x) \xrightarrow[]{\mathcal{T}}}_{\mathcal{G}_1}  \mathbb{R}^p \xrightarrow[]{\mathbf{f}(y)*} u(y). $$

Various neural operators referred in~\cite{DBLP:journals/jmlr/KovachkiLLABSA23} could be chosen as the operator $\mathcal{G}_0$. The parameterisation of the operator $\mathcal{G}_1$ will be introduced in Section~\ref{sec:implementation_RDO}. In addition, the comparison of RDO with DeepONet and FNO is introduced in Table~\ref{tab:comparison}.
\begin{table}[H]
    \centering
    \caption{Comparison between DeepONet, FNO and RDO}
    \label{tab:comparison}
    \begin{tabular}{cccc}
         & RDO & DeepONet & FNO \\
         \hline
         Input domain $\mathcal{D}_1$ & \textbf{ Arbitrary}  & Arbitrary& Regular \\
         Output domain $\mathcal{D}_2 $& \textbf{Arbitrary} & Arbitrary& Regular\\
        The relation between $\mathcal{D}_1$ and $\mathcal{D}_2$ & \textbf{Arbitrary}  & Arbitrary & Identicl\\
        Prediction location & \textbf{Arbitrary}   & Arbitrary& Grid points \\
        \textbf{Resolution Invariant} & \textbf{Yes} & No & Yes\\
        Time Dependent Problems & \textbf{Yes}  & Yes & No \\
    \end{tabular}   
\end{table}

\subsection{Approximation theorem} 
In this part, we introduce the generalized universal approximation of RDO.
Let $C(\mathcal{D})$ denote a Banach space of all continuous functions defined on the bounded domain $\mathcal{D}$ with the norm $||f||_{C(\mathcal{D})}=\max_{x\in \mathcal{D}}|f(x)|$. 

\begin{lemma}\label{lemma:discretization_operator}
    Let $f \in C(\mathcal{D})$ be a continuous function defined on the domain $\mathcal{D}$ and $x_1, \dots, x_m$ be $m$ fixed points in $\mathcal{D}$. The vector $f_m = (f(x_1), \dots, f(x_m))^T \in \mathbb{R}^m$ is obtained by evaluating the function $f$ at these fixed points. Then, the mapping $\mathcal{Z}: C(\mathcal{D}) \rightarrow \mathbb{R}^m$ with $\mathcal{Z}(f) = f_m$ is a bounded and continuous linear operator.
\end{lemma}

\begin{proof}
 $\forall f^1,f^2\in C(\mathcal{D}) $, we have
    $$\mathcal{Z} f^1 + \mathcal{Z} f^2 = 
    \left[\begin{array}{c}
        f^1(x_1) + f^2 (x_1)\\
        f^1(x_2) + f^2 (x_2)\\
        \vdots\\
        f^1(x_m) + f^2 (x_m)
    \end{array}\right]
    = \mathcal{Z}(f^1+f^2).$$
    And for any $f\in  C(\mathcal{D})$, it is easy to verify that
    $$\mathcal{Z}(\alpha f) = \alpha \mathcal{Z}(f)$$
    where $\alpha$ is a scalar. Therefore, $\mathcal{Z}$ is a linear operator. In addition,
    $$ ||\mathcal{Z}(f)||_\infty =  ||f_m ||_\infty = \max_{i} |f(x_i)| \leq \max_{x\in \mathcal{D}}|f(x)| =  ||f||_{C(\mathcal{D})} ,$$
    which implies the operator $\mathcal{Z}$  is  bounded. 

    Suppose that $||\mathcal{Z}||\leq M$ where $M$ is a constant, for any $\epsilon>0$, there exists $f^1,f^2\in C(\mathcal{D})$ such that $||f_1 - f_2||\leq \frac{\epsilon}{M}$. Then,
    \[
    ||\mathcal{Z}f^1- \mathcal{Z} f^2 || \leq  || \mathcal{Z} || ||f^1- f^2 || \leq  M \cdot \frac{\epsilon}{M} = \epsilon,
    \]
which implies that the operator $\mathcal{Z}$ is also continuous. 
\end{proof}
Note that the operator $\mathcal{Z}$ is called ``encoder" in \cite{lanthaler2022error}. Next, we first introduce the generalized universal approximation theorem of DeepONet and then give the corresponding approximation theorem of RDO.

\begin{lemma}\label{lemma:deeponet}
    (\cite{Lu2021}). Suppose that $X$ is a Banach space, and $\mathcal{D}_1 \subset X$, $\mathcal{D}_2 \subset \mathbb{R}^{d_2}$ are two compact sets in $X$ and $\mathbb{R}^{d_2}$, respectively. Let $\mathcal{A}$ be a compact set in $C\left(\mathcal{D}_1\right)$, and assume that $G: \mathcal{A} \rightarrow C\left(\mathcal{D}_2\right)$ is a nonlinear continuous operator. Then, for any $\epsilon>0$, there exist positive integers $m$ and $p$, continuous vector functions $\mathbf{b}: \mathbb{R}^m \rightarrow \mathbb{R}^p, \mathbf{f}: \mathbb{R}^{d_1} \rightarrow \mathbb{R}^p$, and $\lbrace  x_i \in \mathcal{D}_1 \rbrace_{i=1}^m $, such that
$$
\left|G(a)(y)-\langle\underbrace{\mathbf{b}\left(a\left(x_1\right), a\left(x_2\right), \cdots, a\left(x_m\right)\right)}_{\text {branch }}, \underbrace{\mathbf{f}(y)}_{\text {trunk }}\rangle\right|<\epsilon
$$
holds for all $a \in \mathcal{A}$ and $y \in \mathcal{D}_2$, where $\langle\cdot, \cdot\rangle$ denotes the dot product in $\mathbb{R}^p$. 
\end{lemma}

\begin{theorem}\label{theorem:1}
     Suppose that $X$ is a Banach space, $\mathcal{D}_1 \subset X$, $\mathcal{D}_2 \subset \mathbb{R}^{d_2}$ are two compact sets in $X$ and $\mathbb{R}^{d_2}$, respectively. Let $\mathcal{A}$ be a compact set in $C\left(\mathcal{D}_1\right)$, and assume that $G: \mathcal{A} \rightarrow C\left(\mathcal{D}_2\right)$ is a nonlinear continuous operator.   Then, for any $\epsilon>0$, there exist a positive integer $p$, a continuous  operator $\mathcal{G}_1: \mathcal{A} \rightarrow \mathbb{R}^p$, and a continuous vector function $\mathbf{f}: \mathbb{R}^{d_u} \rightarrow \mathbb{R}^p$, such that
$$
\left|G(a)(y)-\langle\mathcal{G}_1\left(a\right), \mathbf{f}(y)\rangle\right|<\epsilon,
$$
holds for all $a \in \mathcal{A}$ and $y \in \mathcal{D}_2$, where $\langle\cdot, \cdot\rangle$ denotes the dot product in $\mathbb{R}^p$. 
\end{theorem}
\begin{proof}
    According to Lemma \ref{lemma:deeponet}, there exist two continuous functions $g:\mathbb{R}^m\rightarrow \mathbb{R}^p$, $f:\mathbb{R}^{d_1}\rightarrow \mathbb{R}^{p}$ such that 
    $$ 
    \left|G(a)(y)-\langle g(a(x_1), a(x_2), \cdots, a(x_m)), \mathbf{f}(y)\rangle\right|<\epsilon,
    $$
    where $\lbrace  x_i \in \mathcal{D}_1 \rbrace_{i=1}^m $ represent the set of $m$ points in $\mathcal{D}_1$.
    Let $\mathcal{G}_1 = g \circ \mathcal{Z} $, where $\mathcal{Z}$ is the discretization operator introduced in Lemma~\ref{lemma:discretization_operator}.
     Thus, we obtain following inequality,
     $$ \left|G(a)(y)-\langle g\circ \mathcal{Z}(a), \mathbf{f}(y)\rangle\right|= \left|G(a)(y)-\langle  \mathcal{G}_1(a), \mathbf{f}(y)\rangle\right|<\epsilon.
    $$
    
\end{proof}

Next, we introduce the universal approximation theorem of our RDO, where the key is that we use a neural operator to approximate the continuous operator $\mathcal{G}_1$ of the branch net of RDO. We follow the notations in \cite{DBLP:journals/jmlr/KovachkiLLABSA23}, and let $\displaystyle{\mathcal{NO}_n}$ denote a set of  $n$-layered neural operators.
For simplicity,  we denote $\mathcal{NO}$ as $\displaystyle{\cup_{i=2}^{+\infty} \mathcal{NO}_i}$. Next, we give the theorem on the universal approximation error of neural operator. 

\begin{lemma}\label{lemma:neural_operator}
    (\cite{DBLP:journals/jmlr/KovachkiLLABSA23}) Let $G_1: C(\mathcal{D}_1) \rightarrow C(\mathcal{D}_2)$ be a continuous operator, then for any compact set $\mathcal{A} \subset C(\mathcal{D}_1)$ and $0<\epsilon \leq 1$, there exists a neural operator $\mathcal{G} \in \mathcal{NO}$ such that
    \[
    \sup _{a \in \mathcal{A}}\left\|G_1(a)-\mathcal{G}(a)\right\|_{C(\mathcal{D}_2)}\leq \epsilon.
    \]
    \end{lemma}

\begin{lemma}\label{lemma:integral_operator_linear_bounded}
Let $C(\mathcal{D},\ \mathbb{R}^p)$ be a Banach space of all vector-valued continuous functions defined on the bounded domain $\mathcal{D}$ with the norm $||f||_{C(\mathcal{D},\mathbb{R}^p)}=\max_{x\in \mathcal{D}}||f(x)||_\infty$, where $f : x\in\mathcal{D}\rightarrow \mathbb{R}^p$. Then, the integral operator $\mathcal{T}$ defined in Equation~\eqref{eq:integral_equation} is a linear and bounded operator on $C(\mathcal{D},\ \mathbb{R}^p)$.
\end{lemma}
   \begin{proof}
       For any $f,\ g\in C(\mathcal{D},\mathbb{R}^p)$ and scalars $\alpha,\ \beta$, we have
       \begin{align*}
        \mathcal{T}(\alpha f+\beta g) = \begin{bmatrix}
            \int_{D} (\alpha f_1(x)+\beta g_1(x)) dx \\
            \int_{D} (\alpha f_2(x)+\beta g_2(x)) dx \\
            \vdots\\
            \int_{D} (\alpha f_p(x)+\beta g_p(x)) dx
            \end{bmatrix}  
            = \alpha \mathcal{T}(f)+\beta \mathcal{T}(g).
       \end{align*}
       Thus, the operator $\mathcal{T}$ is a linear operator. 
       
       In addition, for any $f\in C(\mathcal{D},\ \mathbb{R}^p)$, we have
       \begin{align*}
        ||\mathcal{T}(f)||_\infty = \max_i |\int_{D} f_i(x) dx| &\leq   \max_i \int_{D} |f_i(x)| dx \leq \max_i  \max_{x\in\mathcal{D}} ( |D|  \cdot |f_i(x)| )\\
        &=|D| \max_{x\in\mathcal{D}} \max_i |f_i(x)|\\
        &= |D| \max_{x\in\mathcal{D}} ||f(x)||_\infty\\
        &= |D| \cdot ||f||_{C(\mathcal{D},\mathbb{R}^p)},
       \end{align*}
       which implies that the operator $\mathcal{T}$ is a bounded linear operator.
   \end{proof}

\begin{lemma}\label{lemma:integral_operator}
     Let $G_1: C(\mathcal{D}_1) \rightarrow C(\mathcal{D}_2)$ be a continuous operator, then for any compact set $\mathcal{A}\subset C\left(\mathcal{D}_1\right)$ and $0<\epsilon\leq\min(4,\ 4||\mathcal{T}||)$, there exists a neural operator $\mathcal{G}_0\in \mathcal{NO}$ such that 
    \[
    \sup_{a\in \mathcal{A}}  ||\mathcal{T}\circ\mathcal{G}_0(a)-G_1(a)|| \leq \epsilon,
    \]
where $\mathcal{T}$ is the integral operator defined by Equation~\eqref{eq:integral_equation} and $||\mathcal{T}||$ is its corresponding norm on $C\left(\mathcal{D}_1\right)$.
\end{lemma}
\begin{proof}
According to Lemma~\ref{lemma:neural_operator}, there exists a neural operator $\mathcal{G}$ that satisfies 
\begin{equation}\label{eqn:1b}
\sup_{a\in \mathcal{A}} ||\mathcal{G}(a) - G_1(a)|| \leq \frac{\epsilon}{4},
\end{equation}
with $0<\epsilon\leq 4$.
    Then,
    \begin{align*}
        ||\mathcal{T}\circ\mathcal{G}_0(a)-G_1(a)|| &= || \mathcal{T}\circ\mathcal{G}_0(a)-\mathcal{G}(a)+\mathcal{G}(a)-G_1(a) ||\\
        &\leq || \mathcal{T}\circ\mathcal{G}_0(a)-\mathcal{G}(a)|| + ||\mathcal{G}(a)-G_1(a) ||.
        \end{align*}
    We can thus obtain that   
    \begin{align*}
        \sup_{a\in \mathcal{A}}  ||\mathcal{T}\circ\mathcal{G}_0(a)-G_1(a)|| \leq \frac{\epsilon}{4} + \sup_{a\in \mathcal{A}} || \mathcal{T}\circ\mathcal{G}_0(a)-\mathcal{G}(a)||.
    \end{align*}
    Since $\mathcal{T}$ is a bounded linear operator, for the continuous operator $G_1$, there exists an operator $G$ such that 
    \[
    ||\mathcal{T}\circ G(a) - G_1(a) ||\leq \frac{\epsilon}{4},
    \]
and
    \begin{align*}
        || \mathcal{T}\circ\mathcal{G}_0(a)-\mathcal{G}(a)|| &= || \mathcal{T}\circ\mathcal{G}_0(a)-\mathcal{T}\circ G(a) + \mathcal{T}\circ G(a) - \mathcal{G}(a) || \\
        &\leq  || \mathcal{T}\circ\mathcal{G}_0(a)-\mathcal{T}\circ G(a)|| + || \mathcal{T}\circ G(a) - \mathcal{G}(a) ||\\
        &\leq  || \mathcal{T} || ||\mathcal{G}_0(a) - G(a)|| + || \mathcal{T}\circ G(a) - \mathcal{G}(a) ||
    \end{align*}
    Therefore,  
    \[
    || \mathcal{T}\circ G(a) - \mathcal{G}(a) || \leq ||
     \mathcal{T}\circ G(a) - G_1(a) ||+||G_1(a) -\mathcal{G}(a)|| \leq \frac{\epsilon}{4} + \frac{\epsilon}{4} =\frac{\epsilon}{2},
     \]
     where the second inequality is obtained via Equation~\eqref{eqn:1b}.
    Thus,
    \[
    \sup_{a\in \mathcal{A}}  ||\mathcal{T}\circ\mathcal{G}_0(a)-G_1(a)|| \leq \frac{3\epsilon}{4} + ||\mathcal{T}|| || \mathcal{G}_0(a)  - G(a)||. 
     \]
    According to Lemma~\ref{lemma:neural_operator}, there exists a neural operator $\mathcal{G}_0(a)$ with $0<\epsilon\leq4\|\mathcal{T}\|$ such that 
    \[
     \|\mathcal{G}_0(a) - G(a)\| \leq \frac{\epsilon}{4 ||\mathcal{T}||},
     \]
      which in turn yields,
      \[
       \sup_{a\in \mathcal{A}}  ||\mathcal{T}\circ\mathcal{G}_0(a)-G_1(a)|| \leq \epsilon.
      \]
\end{proof}

As pointed out in~\cite{DBLP:journals/jmlr/KovachkiLLABSA23}, the results given by Lemma~\ref{lemma:neural_operator} can be straightforward generalized to vector-valued settings. Therefore, our results in Lemma~\ref{lemma:integral_operator} also generalize straightforward to vector-valued settings.

\begin{theorem} \label{theorem:2}
    Let $X$ be a Banach space, $\mathcal{D}_1 \subset X$ and $\mathcal{D}_2 \subset \mathbb{R}^{d_2}$ be two compact sets in $X$ and $\mathbb{R}^{d_2}$, respectively. Given  a nonlinear continuous operator $G: \mathcal{A} \rightarrow C\left(\mathcal{D}_2\right)$, where  $\mathcal{A}$ is a compact set in $C\left(\mathcal{D}_1\right)$. For any $0<\epsilon \leq \min\{4M,\ 4M\|\mathcal{T}\|\}$, there exists a positive integer $p$, a continuous function $\mathbf{f}: \mathbb{R}^{d_2} \rightarrow \mathbb{R}^p$, and a continuous operator $\mathcal{G}_0 \in \mathcal{NO}$ such that we can construct a continuous operator $\mathcal{T}\circ\mathcal{G}_0: \mathcal{A} \rightarrow \mathbb{R}^p$ with
\[
\sup_{a\in \mathcal{A}}  \left|G(a)(y)-\langle\mathcal{T}\circ\mathcal{G}_0\left(a\right), \mathbf{f}(y)\rangle\right|<\epsilon,
\]
holds for all $a \in \mathcal{A}$ and $y \in \mathcal{D}_2$. Here $M$ represents the upper bound of $\mathbf{f}$ in $\mathcal{C}(D_2)$, $\mathcal{T}$ is the integral operator in Equation~\eqref{eq:integral_equation}, and $\langle\cdot, \cdot\rangle$ denotes the dot product in $\mathbb{R}^p$. 
\end{theorem}

\begin{proof}
    According to Theorem \ref{theorem:1}, for any $\epsilon_1 >0$, there exist a continuous operator ${G}_1$ and a continuous function $\mathbf{f}$ such that
    \[
    \left|G(a)(y)-\langle G_1(a), \mathbf{f}(y) \rangle\right|\leq \epsilon_1.
    \]
    Then,
    \begin{align*}
      \left|G(a)(y)-\langle\mathcal{T}\circ\mathcal{G}_0\left(a\right), \mathbf{f}(y)\rangle\right| &\leq 
        \left|G(a)(y)-\langle G_1(a), \mathbf{f}(y) \rangle\right|+\left|\langle G_1(a), \mathbf{f}(y) \rangle-\langle\mathcal{T}\circ\mathcal{G}_0\left(a\right), \mathbf{f}(y)\rangle\right| \\
        &\leq  \epsilon_1 +\left|\langle G_1(a)-\mathcal{T}\circ\mathcal{G}_0\left(a\right), \mathbf{f}(y) \rangle\right| \\
        &\leq \epsilon_1 + || G_1(a)-\mathcal{T}\circ\mathcal{G}_0\left(a\right)|| ||\mathbf{f}  ||.
    \end{align*}
Therefore,
\[
\sup_{a\in \mathcal{A}}  \left|G(a)(y)-\langle\mathcal{T}\circ\mathcal{G}_0\left(a\right), \mathbf{f}(y)\rangle\right| \leq  \epsilon_1 + ||\mathbf{f}  || \cdot \sup_{a\in \mathcal{A}}  || G_1(a)-\mathcal{T}\circ\mathcal{G}_0\left(a\right)||. 
\]
Let $||\mathbf{f}||= \max_{y\in \mathcal{D}_2} |\mathbf{f}(y)|$ be bounded by a constant $M$. According to Lemma~\ref{lemma:integral_operator}, there exists a continuous operator 
 $\mathcal{G}_0\in \mathcal{NO}$ such that 
 \[
 ||G_1(a)-\mathcal{T}\circ\mathcal{G}_0\left(a\right)|| \leq \frac{(\epsilon-\epsilon_1)}{M},
 \]
for $\epsilon_1<\epsilon\leq \min\{4M,\ 4M\|\mathcal{T}\|\}$.
Finally, we obtain
\[
\left|\mathcal{G}(a)(y)-\langle\mathcal{T}\circ\mathcal{G}_0\left(a\right), \mathbf{f}(y)\rangle\right| \leq \epsilon_1 + \frac{(\epsilon-\epsilon_1)}{M} M =\epsilon.
\]
\end{proof}


\subsection{Parameterization and implementation of RDO}\label{sec:implementation_RDO}
In RDO, we construct parametric models to approximate the operator $\mathcal{G}_0$ and the vector-valued function $\mathbf{f}$. Denote their corresponding parametric models by $\mathcal{G}_{0,\theta}$ and $\mathbf{f}_\theta$, respectively, where the subscript $\theta$ represents the parameters. 
In this paper, $\mathbf{f}_\theta$ is simply parameterized by FNN and we denote RDO by $\mathcal{G}:=\langle\mathcal{T}\circ\mathcal{G}_{0,\theta}\left(a\right), \mathbf{f}_\theta(y)\rangle$, where $\mathcal{G}_{0,\theta}$ is parameterized by multiple  parametric kernel integral transformation. 
Here, the kernel integral transformation $\mathcal{K}_\theta$ is given by 
\begin{equation}
    z(x) = \mathcal{K}_\theta(\phi)(x) =\int_\mathcal{D} \kappa_\theta(x,y)\phi(y)dy, \qquad x\in \mathcal{D} \label{eq:kernel_integration}
\end{equation}
where $\kappa_\theta(\cdot,\cdot)$ is the kernel function. Since it is intractable to compute the integral transformation directly, there are various algorithms proposed to approximate this kernel integral transformation~\cite{DBLP:journals/jmlr/KovachkiLLABSA23,DBLP:journals/jmlr/OngSY22}. 
Here, we introduce two widely used approaches, which are the keys of FNO and the attention mechanisms~\cite{vaswani2017attention, DBLP:journals/corr/BahdanauCB14}.

\subsubsection{Fourier integral operator}\label{sec:fourier_integral_operator}

By rewriting $\kappa_\theta(x,y)$ as $\kappa_\theta(x-y)$, Equation~\eqref{eq:kernel_integration} represents the convolution of $\phi(y)$ which could be computed using the fast Fourier transformation (FFT) represented by $\mathcal{F}\{\cdot\}$ and the inverse fast Fourier transformation (iFFT) represented by $\mathcal{F}^{-1}\{\cdot\}$. The process is given by
\begin{equation}\label{eqn:fft}
\mathcal{F}\{\mathcal{K}_\theta(\phi)\} = \mathcal{F}\{\int_\mathcal{D} \kappa_\theta(x-y)\phi(y)dy\} = \mathcal{F}\{\kappa_\theta(x-y)\}\diamond\mathcal{F}\{\phi(y)\},
\end{equation}
where $\diamond$ refers to the convolution operation. We can compute the result of Equation~\eqref{eq:kernel_integration} by applying the iFFT to Equation~\eqref{eqn:fft}, which gives
$$\mathcal{K}_\theta(\phi)(x) = \mathcal{F}^{-1}\big\{\mathcal{F}\{\kappa_\theta(x-y)\}\diamond\mathcal{F}\{\phi(y)\}\big \}.$$

In FNO, the kernel function $ \kappa_\theta$ is directly parameterized in the frequency domain and denote $\mathcal{F}\{\kappa_\theta(x-y)\}$ as $\mathcal{R}_\theta$. The complex-valued vector $\mathcal{F}\{\phi(y)\}(k)\in \mathbb{C}^{d_\phi }$ refers to  the $k$-th frequency response mode of the input function $\phi(\cdot)$. 
The complex-valued matrix $\mathcal{R}_\theta(k)\in \mathbb{ C}^{d_\phi \times d_\phi}$  represents the $k$-th frequency of the kernel function $\mathcal{F}\{\kappa_\theta(\cdot)\}(k)$. 
For efficient computations, the frequencies higher than $k$ are truncated. 
To retain the high frequency information of the input function $\phi(\cdot)$ and the kernel function $\kappa_\theta(\cdot)$, a residual term is added. Finally, the Fourier integral operator (FIO) is given by 
\[
z(x) = \hat{\mathcal{F}}(\phi)(x) := \sigma\textbf{(} W_\theta \phi(x) +   \mathcal{K}_\theta(\phi)(x)\textbf{)} ,
\]
where $\sigma(\cdot)$ is a nonlinear activation function and $\mathcal{\widehat{F}}$ represents the approximated transformation of FIO. This procedure is illustrated by Figure~\ref{fig:FIO}.
\begin{figure}[H]
    \centering
    \begin{center}
            \includegraphics[width=0.6\linewidth]{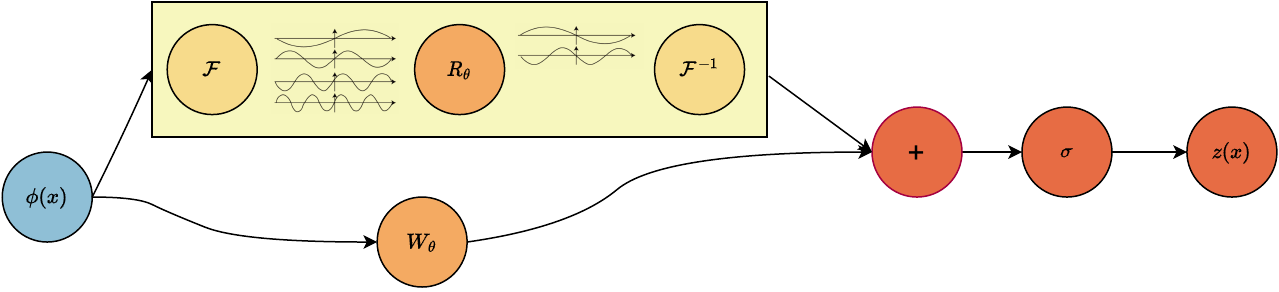}
    \end{center}
    \vspace{-.2cm}
    \caption{FIO architecture.}\label{fig:FIO}
\end{figure}

\subsubsection{Attention integral operator}
The attention mechanism has demonstrated its advantage over other deep learning methods in Natural Language Processing. Unlike the traditional sequence models which always lose the information from a long time ago, such as RNN, LSTM etc., its scaled dot-product attention mechanism can unearth the latent information of the total sequence. As pointed by~\cite{cao2021choose}, the Attention mechanism is equivalent to a variant of neural operator. Furthermore, the attention mechanism without softmax transformation has a resemblance with a Fourier-type kernel integral transform~\cite{abfa1ea9-6eb2-3147-9b71-463685fe302c}. \citet{DBLP:journals/jmlr/KovachkiLLABSA23} explained that the popular transformer architecture is a specific kind of neural operator, named Attention Integral Operator (AIO) in this work.

Let $q_\theta, k_\theta$ and $v_\theta$ be the point-wise nonlinear transformations for three feature mappings $q,\ k,\ v:\mathcal{D}\rightarrow \mathbb{R}^{d}$, where $q(x):= q_\theta(\phi(x)), k(x):= k_\theta(\phi(x))$ and $v(x) := v_\theta(\phi(x))$ for each $x\in \mathcal{D}$ and the input function $\phi(x)$. The continuous version of the transformer is given by
  \begin{align}\label{eq:AIO}
    z(x) =  \frac{\int_\mathcal{D}\exp\{q(x)\cdot k(y)\} v(y)dy}{\int_\mathcal{D}\exp\{q(x)\cdot k(y)\}dy} =\int_\mathcal{D}\frac{\exp\{q(x)\cdot k(y)\}}{\int_\mathcal{D}\exp\{q(x)\cdot k(y)\}dy} v(y)dy.
  \end{align}
Let $(\phi(x_1),\phi(x_2),\dots,\phi(x_m))^T\subset \mathbb{R}^{m\times d_\phi}$ be the input sequences where $m$ is the number of grid points in the domain $\mathcal{D}$. Then, Equation~\eqref{eq:AIO} is approximated  through the following summation (see Figure~\ref{fig:AIO}),
\[
z(x_i)\approx \sum^{m}_{j=1} \frac{\exp\{q_i \cdot k_j\}\cdot v_j}{\sum^{m}_{l=1} \exp\{q_i\cdot k_l\}},
\]
where $q_i,\ k_i,\ v_i$ are $q(x),k(x)$ and $v(x)$ being evaluated at $x=x_i$, respectively. 
\begin{figure}[H]
    \begin{center}
            \includegraphics[width=0.6\linewidth]{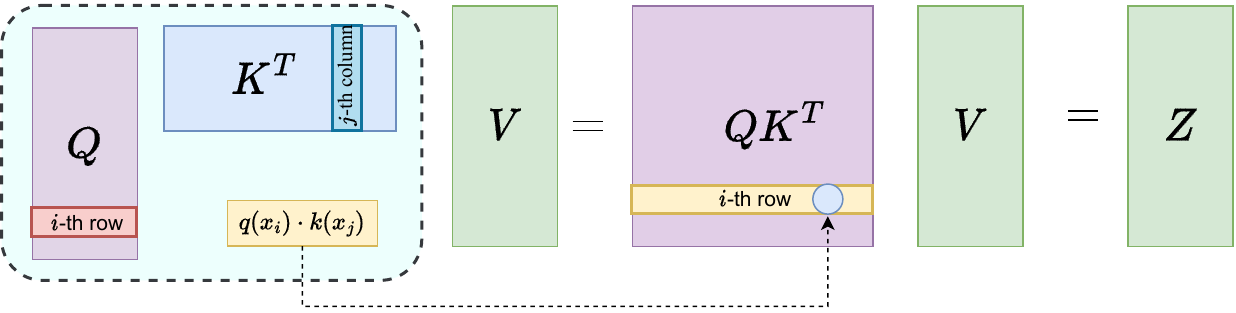}
    \end{center}
    \vspace{-.2cm}
    \caption{AIO architecture.}\label{fig:AIO}
\end{figure}


\subsection{Architecture of RDO}

In this subsection, we provide a feasible scheme of RDO where FIO and AIO are used to approximate the nonlinear operator $\mathcal{G}_{0}$. The complete structure of RDO is shown in Figure~\ref{fig:structure}. Here, $T_1$ is the number of FIO layers and $T_2$ is the number of AIO layers. $\mathcal{P}_0$, $\mathcal{P}_1$, and $\mathcal{P}_2$ are point-wise nonlinear transformations. 
\begin{figure}[H]
    \begin{center}\includegraphics[width=.7 \linewidth]{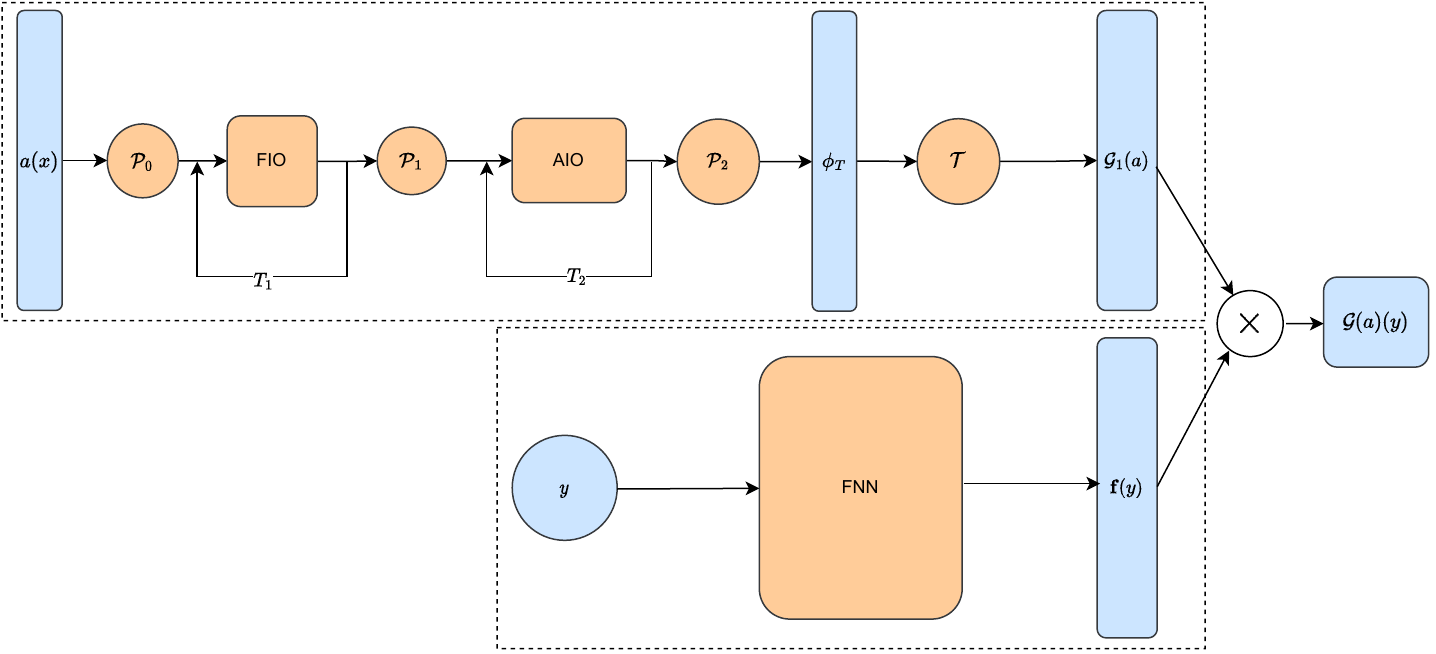}
        \end{center}
        \caption{RDO architecture.}\label{fig:structure}
    \end{figure}
    
Let $d_{\phi_i}$ represent the dimension of the output function $\phi_i$ of the $i$-th layer where $i\in\{0,\dots,T_1,T_1+1,\dots,T_1+T_2+2\}$. Specifically, $\phi_0 =\mathcal{P}_0(a(x))\in\mathbb{R}^{d_{\phi_0}}$, $\phi_{T_1} =\hat{\mathcal{F}}(\phi_{T_1-1}(x))\in\mathbb{R}^{d_{\phi_{T_1}}}$, and $\phi_{T}=\mathcal{G}_0(a(x)) =\mathcal{P}_{2}(\phi_{T_1+T_2+1})\in\mathbb{R}^{d_{\phi_{T_1+T_2+2}}}$ with $T=T_1+T_2+2$.  In practical implementations, $d_{\phi_0}=\dots=d_{\phi_{T_1+1}} =\dots=d_{\phi_{T_1+T_2+1}}$ and $p=d_{\phi_{T_1+T_2+2}}$ are the common choices. The integral transformation $\mathcal{T}$ is approximated by 
 \[
 \mathcal{T}(\phi_{T}) \approx h \sum_i^m \phi_{T}(x_i), 
 \]
where $h$ is the size of the uniform discretization $\{x_i\}_{i=1}^m\in \mathcal{D}_1$. Finally, the corresponding loss function is computed via
\[
\mathbbm{E}_{a\sim\pi}[C({G}(a),\mathcal{G}(a))] \approx \mathbbm{E}_{a\sim\pi}[\frac{1}{n}
\sum_{i=1}^n (G(a)(y_i)-\mathcal{G}(a)(y_i))^2,
\]
where the sequence $\{y_i\}_{i=1}^n$ represents the discretized points in the solution domain $\mathcal{D}_2$.

\section{Numerical experiments}\label{sec:Numerical}

In this section, we use three numerical examples to demonstrate the efficiency of RDO to approximate the PDE solution operator and compare its performance with baseline models, i.e., DeepONet and FNO. Specifically, the branch net and the trunk net of DeepONet are both FNNs. All the networks are constructed by \textit{PyTroch}~\cite{paszke2019pytorch} and trained by the Adam optimizer~\cite{DBLP:journals/corr/KingmaB14} with an initial learning rate of $10^{-3}$. 
In order to find the optimal parameters of the neural networks, the initial learning rate is configured at $0.001$ and decays by 0.5 each 100 epochs over all numerical examples. For fairness, all the parametric models referred in this paper share the same training strategies.

To improve the generalization capability of parametric models, we use the early stopping technique introduced in~\cite{FINNOFF1993771} to train our models, which saves the parameters of neural networks which perform best in the validation set rather than the parameters trained in the final epoch. Meanwhile, the dataset is split into three disjoint subsets, i.e., a training dataset, a validation dataset, and a test dataset. In this paper, the ratio of these datasets is $6:\ 2:\ 2$. 

We utilise the relative $L_2$ norm error to evaluate the performance of different methods, which is given by
\[
RL2E = \frac{||\mathcal{Z}(\hat{u})-\mathcal{Z}(u)||_2}{||\mathcal{Z}(u)||_2}.
\]
Here $\hat{u}$ is the approximate solution, $u$ is the groundtruth, and $\mathcal{Z}$ is the discretization operator in Lemma~\ref{lemma:discretization_operator}. 

\subsection{Stochastic boundary value problem (SVBP)}\label{sec:4_1}
We consider a 1-D elliptic stochastic boundary value problem (SBVP) used as a benchmark of the uncertainty quantification problem~\cite{tripathy2018deep,karumuri2020simulator}, which is given by
\begin{equation}\label{eq_SBVP}
    -\nabla \cdot(a(x, \omega) \nabla u(x, \omega))+c \cdot u(x, \omega)=f,\ \forall x \in [0,1]
\end{equation}
with  $c=15, f=10$. The Dirichlet boundary conditions are given by 
\[
\begin{array}{l}
u(0,\omega)=1, 
u(1,\omega)=0.
\end{array}
\]
The log-normal field is chosen as the input random field $a(\cdot,\cdot) $ given by 
\[
\log (a(x, \omega)) \sim \operatorname{GP}\left(\mu(x), k\left(x, x^{\prime}\right)\right),
\]
with mean  $\mu(x)=0$ and the exponential covariance function $ k\left(x, x^{\prime}\right)$ given by 
\[
k\left(x, x^{\prime}\right)=\sigma^{2} \exp \left\{-\frac{\left|x-x^{\prime}\right|}{\ell_{x}}\right\},
\]
where $\ell_x$ is the correlation length. We set $\ell_x=1$ and $\ell_x=0.3$ to test the performance of our method.

The goal of using neural operators for this problem is to learn the mapping from the spatial-varing function $a(\cdot, \cdot)$ to the solution function $u(\cdot)$. Therefore, the input domain $\mathcal{D}_1$ and the output domain $\mathcal{D}_2$ are the same, i.e., the interval $[0,1]$.  
The Python package \textit{FEniCS}~\cite{logg2012dolfin} is applied to compute 1000 numerical solutions for the input function $a(\cdot, \cdot)$ with resolution 129. Then, we split these samples into three subsets, i.e., 600 training samples, 200 validation samples, and 200 test samples. We use input functions with resolution 33 to train models and subsequently test them on the dataset with resolutions 33, 65, and 129, respectively.

In RDO, we set $T_1$ to 3 and $T_2$ to 1. All the truncated frequencies of three FIOs are 16. Moreover, the number of nodes in each layer of $\mathbf{f}_\theta$ in RDO are $1,\ 100,\ 100,\ 100$, respectively. For DeepONet, the number of nodes for each layer in the  branch net is $33,\ 100,\ 100,\ 100$, respectively. The trunk net shares the same structure with $\mathbf{f}_\theta$ of RDO. For FNO, the modes and the width are set to 16 and 128, respectively.  
\begin{figure}[H]
    \begin{center}
        \subfloat[]{\label{fig:ode_input_function}
            \includegraphics[width=0.4\linewidth]{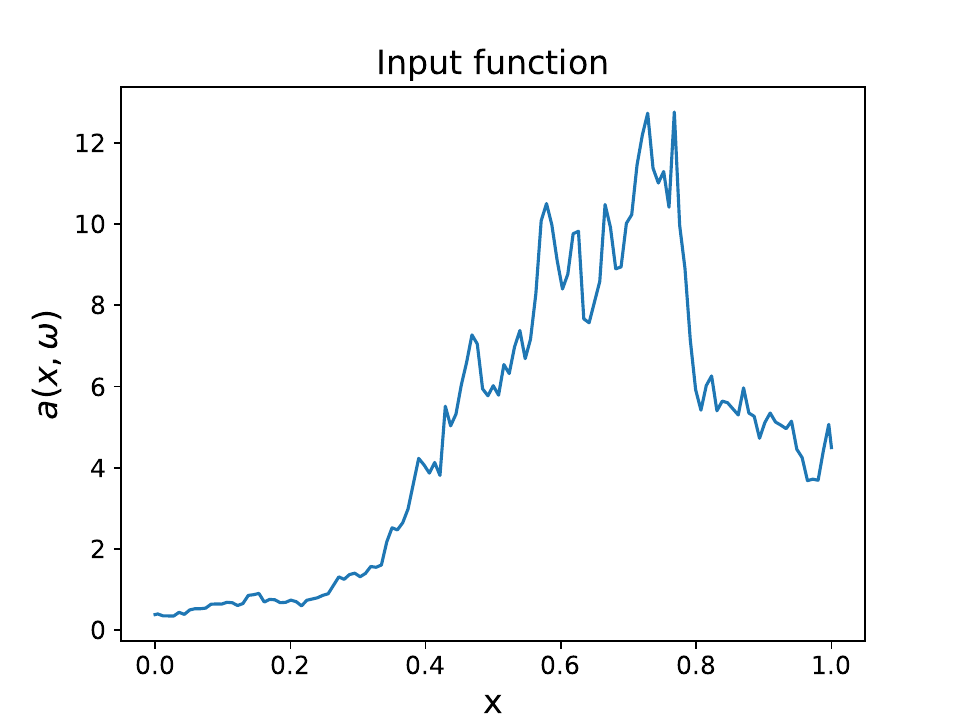}
        }\qquad
        \subfloat[]{
            \includegraphics[width=0.4\linewidth]{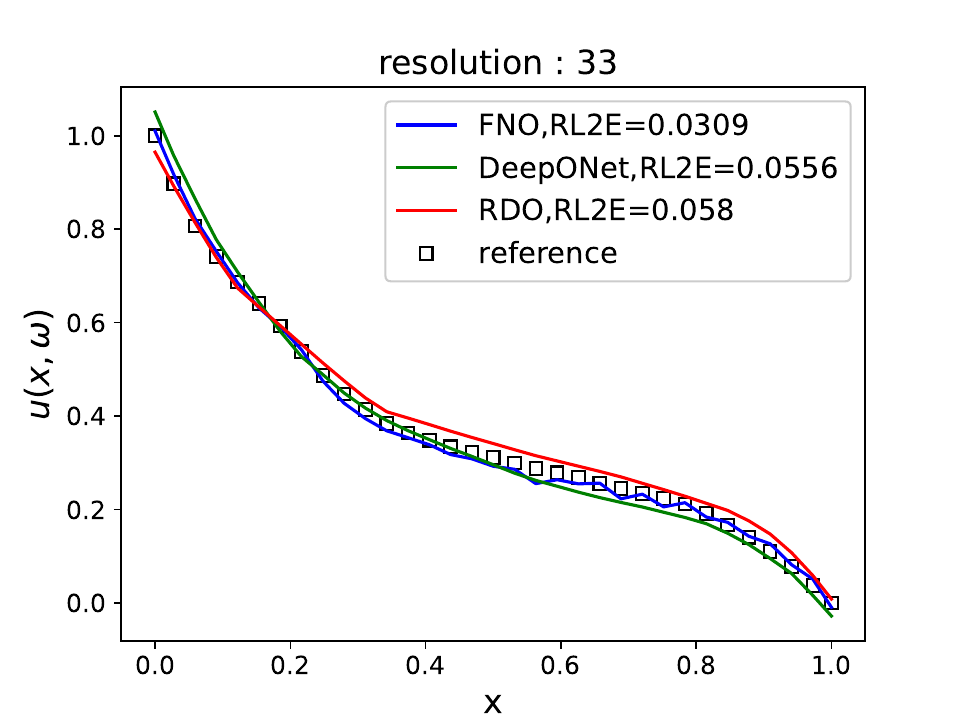}
        }\\
        \subfloat[]{
            \includegraphics[width=0.4\linewidth]{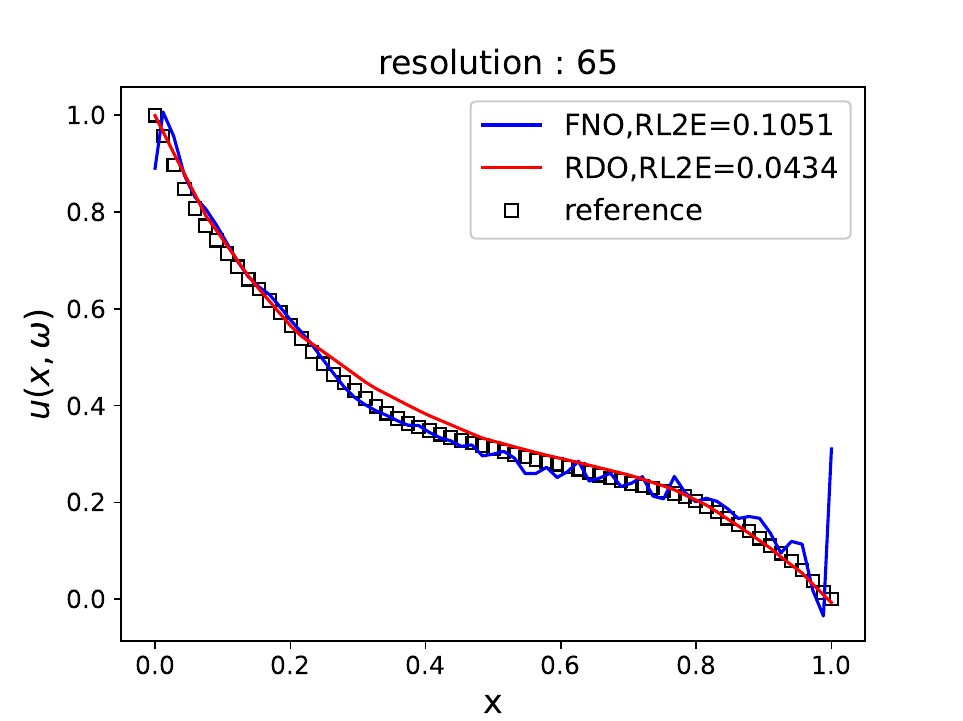}
        }\qquad
        \subfloat[]{
            \includegraphics[width=0.4\linewidth]{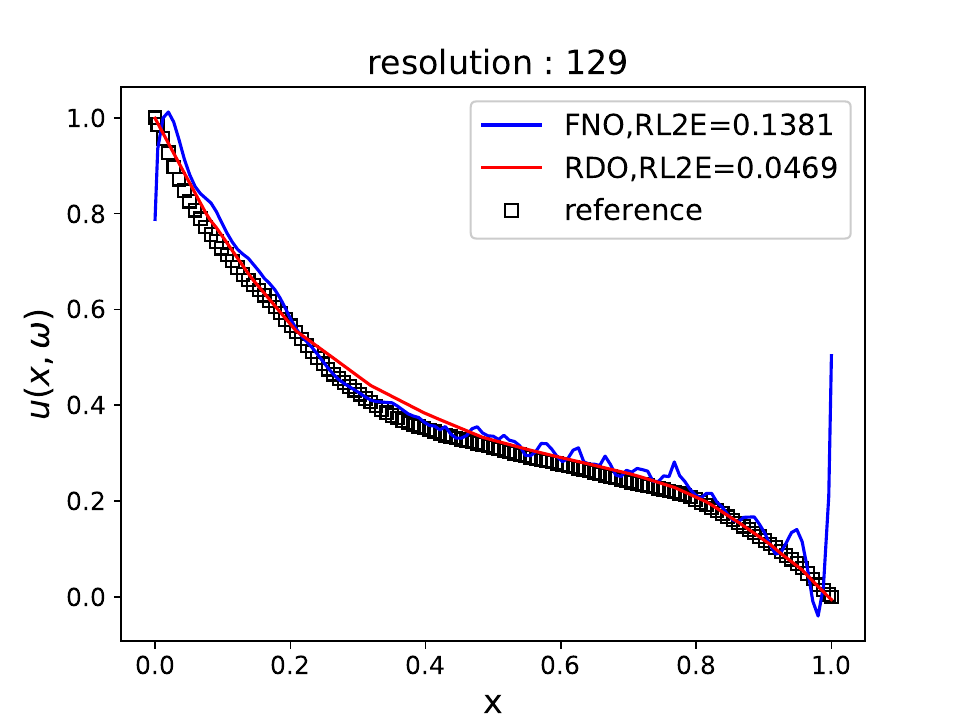}
        }
    \end{center}
    \caption{Computational results for $\ell_x=1$. (a) gives the input function $a(\cdot,\cdot)$.  (b), (c) and (d) shows the prediction results for the resolution 33, 65, 129, respectively. }\label{fig:ODE_comp1}
\end{figure}
Figure~\ref{fig:ODE_comp1} depicts that FNO gets the worst performance on the boundary and RDO could retain the accuracy on higher resolutions. We observe that when testing on higher resolutions, FNO loses more boundary information. This is primarily due to the fact of the inhomogeneous boundary condition of this problem while FFT relies on the homogenous boundary condition. For RDO, this problem is mitigated by introducing the AIO after the FIO layers. 

In Table~\ref{tab:SBVP_comparison}, the results of all methods on the test dataset are summarized. For resolution 33, RDO exhibits intermediate error when compared to other methods. For higher resolution data, RDO gives the best performance. For further demonstration of the effectiveness of RDO, we repeat the above experiment with $\ell_x=0.3$ which corresponds to a more stochastic problem and the results are reported in Table~\ref{tab:SBVP_comparison_lx0.3}. Computational results show that RDO performs best among all methods for even more difficult problems.

\begin{table}[H]
    \centering
    \caption{Relative $L_2$ error for $\ell_x=1$ for different input resolutions.  All models are trained with resolution 33.}
    \label{tab:SBVP_comparison}
    \begin{tabular}{ccccc}
         Resolution & FNO  & RDO & DeepONet \\
         \hline
        33  &  1.13\%  & 1.35\%  & 2.68\% \\
        65 & 7.56\%   & 1.42\% & $-$ \\
        129 & 8.60\% &  1.50\% & $-$ \\
    \end{tabular}
\end{table}
\begin{table}[H]
    \centering
    \caption{Relative $L_2$ error for $\ell_x=0.3$ for different input resolutions. All models are trained with resolution 33.}
    \label{tab:SBVP_comparison_lx0.3}
    \begin{tabular}{ccccc}
         Resolution & FNO  & RDO & DeepONet \\
         \hline
        33  &  1.78\%  & 2.44\%  & 9.53\% \\
        65 & 6.00\%   & 2.32\% & $-$ \\
        129 & 7.16\% &  2.40\% & $-$ \\
    \end{tabular}
\end{table}

\subsection{Darcy flow}

We test the performance of RDO using the 2D Darcy flow problem which is a benchmark in~\cite{li2020fourier},
\begin{align}\label{eq:DarcyFlow}
-\nabla \cdot(K(x,y)\nabla u(x,y))&=f(x,y),\qquad (x,y)\in \Omega\\
u(x,y)&=b(x,y).\qquad (x,y)\in \partial \Omega
\end{align}
Here $K(x,y)$ is the diffusion coefficient function. In this numerical example, we try to showcase the capability of RDO for the problems where the input and output functions are defined on different and irregular domains. Thus, we consider two different geometries of the domain mentioned in~ \cite{lu2022comprehensive}, including a triangular domain in Section ~\ref{p:darcyTraingular} (Case \MakeUppercase{\romannumeral1}) and a triangular domain with notch in Section~\ref{p:darcyTraingulaRDOtch} (Case \MakeUppercase{\romannumeral2}). 
Furthermore, the resolutions of input functions vary, but the interested points of the solution domain are fixed. The datasets are generated by the Partial Differential Equation Toolbox in MATLAB\footnote{MATLAB is a proprietary multi-paradigm programming language and numeric computing environment developed by MathWorks.}. 
It is noteworthy that FNO fails to resolve this problem since the input domain is different from the output domain.

\subsubsection[short]{Darcy problem in a triangular domain}\label{p:darcyTraingular}

In this subsection, we set $\Omega$ as the triangular domain with three vertexes $(0,0),\ (0,1)$, and $(0.5,\sqrt{3/2})$. Set $K(x, y)=0.1$ and $f(x, y)=-1$, and the target of this task is using the neural operator to approximate the mapping from the boundary condition function $b(x, y)$ to the pressure field $u(x, y)$, i.e.,
\[
\mathcal{G}: b(x,y) \rightarrow u(x,y).
\]
Here, we choose a 1-D Gaussian process to generate the boundary condition of the triangular domain, which is given by
\begin{align*}
    a(x) & \sim \mathcal{G} \mathcal{P}\left(0, \mathcal{K}\left(x, x^{\prime}\right)\right), \nonumber\\
    \mathcal{K}\left(x, x^{\prime}\right) & =\exp \left[-\frac{\left(x-x^{\prime}\right)^2}{2 l^2}\right],\ l=0.2.
    \end{align*}

For RDO, we set $T_1=1$, $T_2=1$, and $d_{\phi_0}=p=64$. There are 4 layers in the trunk net of RDO and the width are 2, 128, 128, and 64, respectively. The number of modes kept in FIO is set to 26.
In DeepONet, the structure of the branch net and the trunk net are $[s, 128, 128, 100]$ and $[2, 128, 128,128, 100]$, respectively. Here, $s$ represents the resolution of the input function and the numbers in the bracket represent the width of each layer. In this experiment, the number of interested points of the solution function is 2295. To evaluate the performance of different methods comprehensively, we train RDO on datasets with input resolutions of $51, 101, 201, 401, 801$, and $1601$, respectively. Then, we test these trained models on datasets of various resolutions. However, DeepONet can only be trained and tested using the same resolution datasets.

A representative instance of the solution with the corresponding boundary condition is displayed in Figure~\ref{fig:darcyTraingular}. RDO and DeepONet are trained on the same datasets with resolution of 51, however, RDO can be tested with different resolutions of boundary conditions, i.e., 51 and 101 while DeepONet fails. The corresponding results and error are shown in the two middle columns. The right-most column shows the solution of DeepONet with the boundary condition of resolution 51. We can find that RDO and DeepONet achieve similar accuracy on the boundary function with low resolution. Furthermore, when the resolution of the input function increases, RDO can also maintain a low error level. 
\begin{figure}[H]
    \begin{center}
       \includegraphics[width=0.9\linewidth]{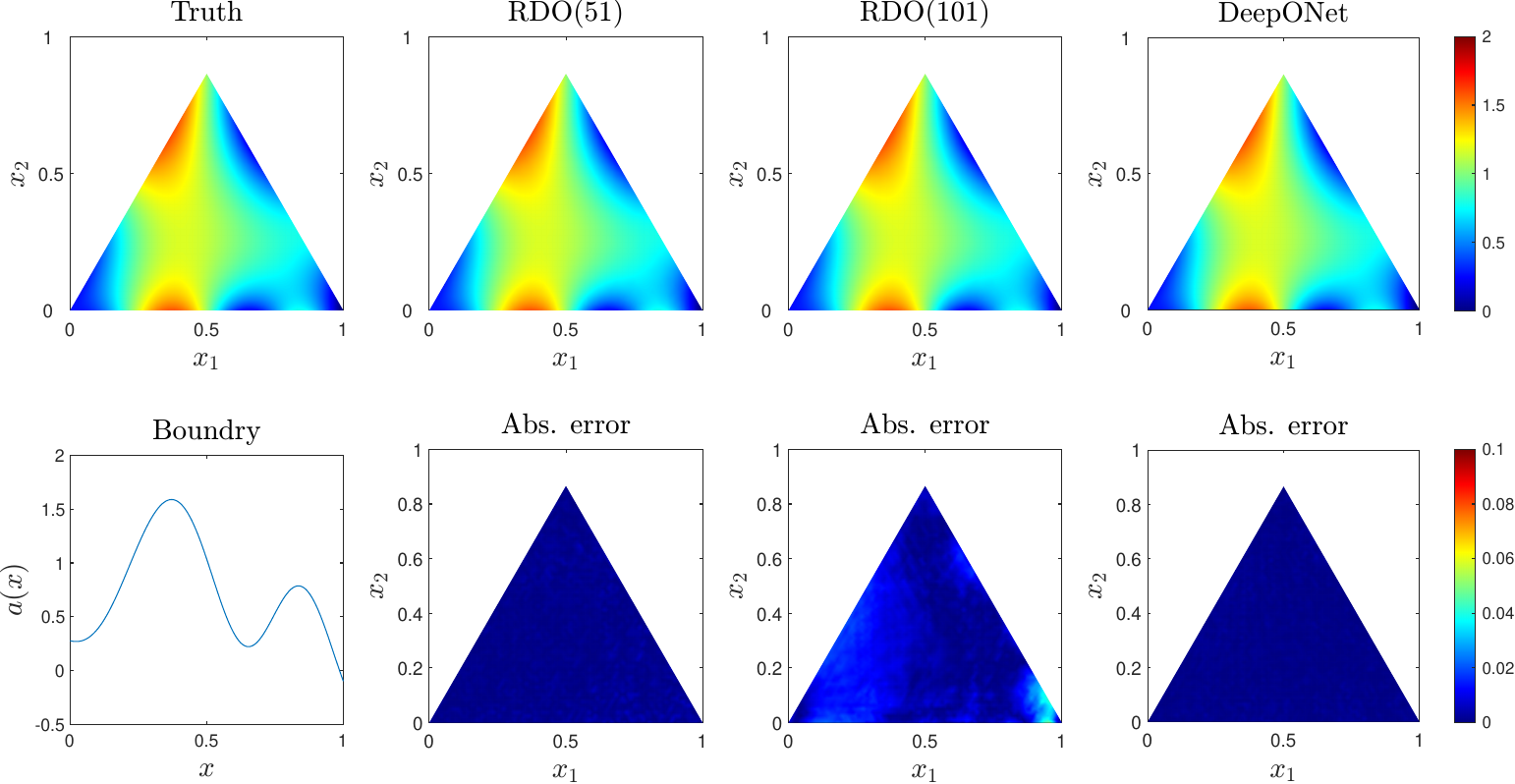}
    \end{center}
    \vspace{-.2cm}
    \caption{One representative computational result for Case \MakeUppercase{\romannumeral1}.} \label{fig:darcyTraingular}
\end{figure}

To assess the learning ability of different methods, we summarize the RL2E of different models which are trained and tested on the same resolution dataset in Table~\ref{tab:darcyTraingular_super}. We can find that with the increase of resolution, the error of DeepONet tends to rise. In contrast, RDO exhibits a consistent error level.
\begin{table}[H]
    \centering
    \caption{The relative $L_2$ error on test data for  Case \MakeUppercase{\romannumeral1} of the Darcy problem. All models are trained and test with the same resolution. }
    \label{tab:darcyTraingular_super}
    \begin{tabular}{cccc}
        Resolution & RDO & DeepONet \\
        \hline
       51  &   $0.21\%$  &  $0.17\%$ \\
       101 &  $0.27\%$ & $0.22\%$\\
       201 &  $0.21\%$ & $0.28\%$\\
       401 &  $0.21\%$ & $0.31\%$\\
       801 &  $0.20\%$ & $0.38\%$\\
       1601 &  $0.19\%$ &  $0.24\%$\\
   \end{tabular}
\end{table}

Table~\ref{tab:darcyTraingular} gives a comprehensive survey of RDO where RDO is trained on each resolution dataset and tested on other resolution datasets. We can easily find that RDO trained on low-resolution dataset achieves low error level for high target resolutions. Moreover, when the resolution of the training data is sufficient high, RDO still can provide accurate predictions for test instances of low resolution. For example, when the training resolution is $1601$, RDO still can keep a comparable low error level for coarser resolution 401.
\begin{table}[H]
    \centering
    \caption{The relative $L_2$ error on test data for Case  \MakeUppercase{\romannumeral1} of the Darcy problem.}
    \label{tab:darcyTraingular}
    \begin{tabular}{ccccccc}
         \diagbox{Test}{Train} & 51 & 101 & 201 & 401  & 801 & 1601\\
         \hline
        51  &   $0.21\%$ & $5.21\%$ & $6.29\%$ & $4.96\%$ & $3.57\%$& $6.18\%$\\
        101 &   $1.67\%$ & $0.27\%$ & $0.91\%$ & $1.17\%$ & $1.38\%$& $1.53\%$\\
        201 &   $2.31\%$ & $1.66\%$ & $0.21\%$ & $0.43\%$ & $0.59\%$& $0.71\%$\\
        401 &   $2.64\%$ & $2.01\%$ & $0.42\%$ & $0.21\%$ & $0.28\%$& $0.35\%$\\
        801 &   $2.81\%$ & $2.07\%$ & $0.58\%$ & $0.28\%$ & $0.20\%$& $0.22\%$\\
        1601 &   $2.91\%$ & $2.12\%$ & $0.67\%$ & $0.35\%$ & $0.22\%$& $0.19\%$\\
    \end{tabular}
\end{table}

\subsubsection[short]{Darcy problem in a triangular domain with notch} \label{p:darcyTraingulaRDOtch}

We introduce an enhanced example by incorporating a notch into the triangular domain. Specifically, the vertices of the notch are located at $(0.49, 0),\ (0.51, 0),\ (0.49, 1)$, and $(0.51, 1)$, respectively.
The boundary conditions are generated in the same manner with Section~\ref{p:darcyTraingular}. A total of 2000 examples of resolution 201 are generated, and these data are divided into three distinct subsets for training, validation, and testing with the corresponding ratio $6: 2: 2$. 
We train RDO on the resolution $51$ dataset and test it on the resolution $51, 101, 201$ datasets to evaluate the ability of zero-shot super-resolution learning. DeepONet is trained and test on the resolution 51 dataset. The structure of RDO and DeepONet are the same with Section~\ref{p:darcyTraingular}.

Figure~\ref{fig:darcyTraingular_notch} depicts the computational results with a corresponding test boundary condition. The predicted solutions and errors of RDO are shown in the two middle columns. The right-most column shows the solution of DeepONet with the boundary condition of resolution 51. We can find that RDO and DeepONet get similar accuracy for the same resolution of the input function. Furthermore, when the resolution of the boundary function increases to 101, RDO can still maintain a low error level.

We summarize the RL2E of different methods in Table~\ref{tab:darcyTraingular_notch}. The error of RDO exhibits an order of magnitude smaller than that of DeepONet for the resolution 51 test dataset. We also find that for super-resolution test datasets, RDO also maintains a low error level. However, DeepONet has no prediction capability for super-resolution test data, since retraining is needed when the dimension of the input changes.
\begin{figure}[H]
    \begin{center}
            \includegraphics[width=0.9\linewidth]{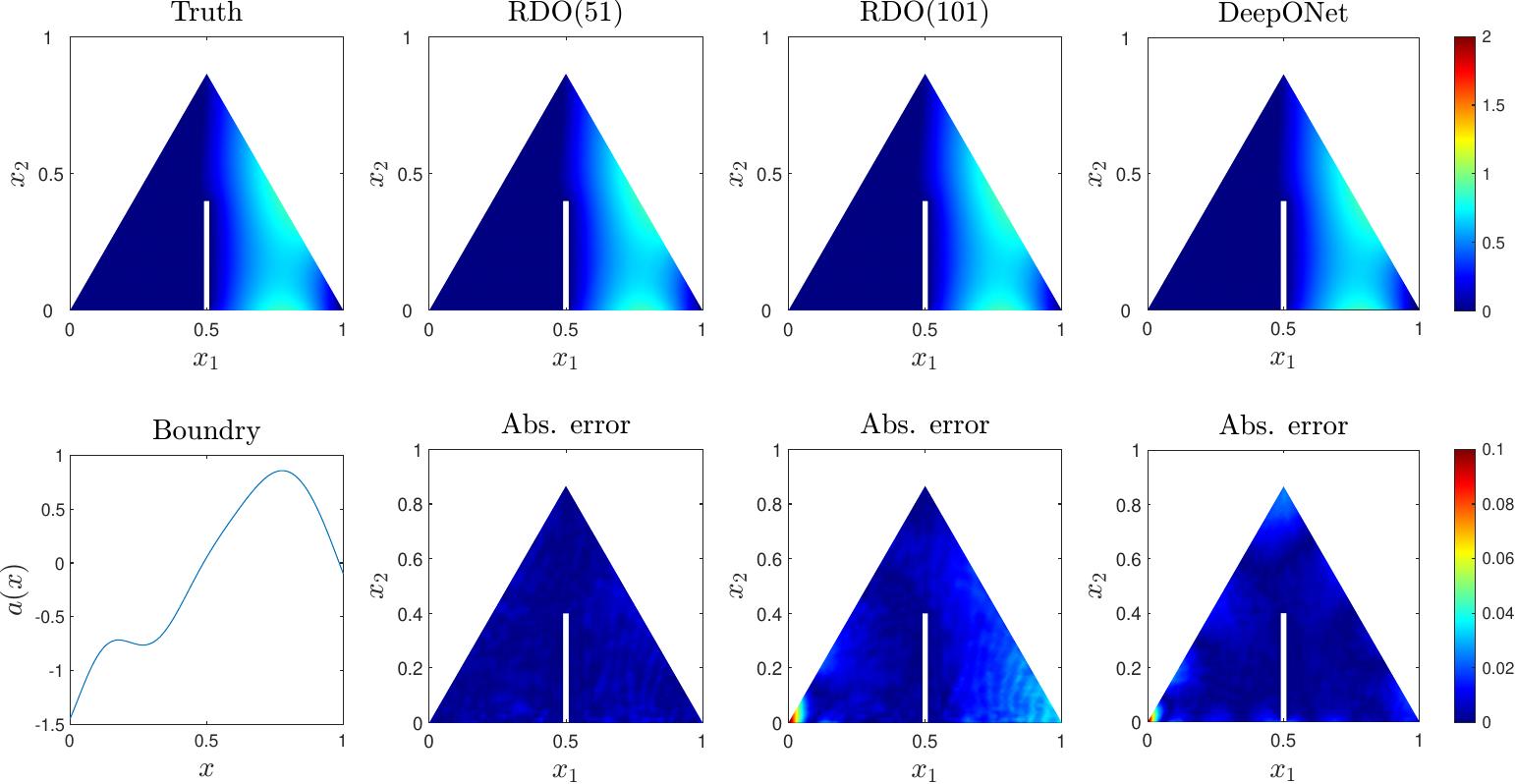}
    \end{center}
    \vspace{-.2cm}
    \caption{Computational results for Case \MakeUppercase{\romannumeral2} of the Darcy problem.} \label{fig:darcyTraingular_notch}
\end{figure}
\begin{table}[H]
    \centering
    \caption{Relative $L_2$ error for Case  \MakeUppercase{\romannumeral2}. All models are trained for input functions of resolution 51.}
    \label{tab:darcyTraingular_notch}
    \begin{tabular}{ccc}
        Resolution  & RDO & DeepONet \\
        \hline
        51  &  $0.31\%$  & $2.68\%$ \\
       101 &$2.99\%$ & $-$ \\
       201 &$3.82\%$ & $-$ \\
   \end{tabular}
\end{table}

\subsection{Burgers' equation}
We consider a one-dimensional Burgers' equation given by,
\begin{align*}
    \frac{\partial u}{\partial t}+u \frac{\partial u}{\partial x}&=v \frac{\partial^2 u}{\partial x^2}, \quad x \in(-1,1),\ t \in(0,1],\\
    u(x,0) &= -sin(\pi\cdot x)\cdot\omega,
\end{align*}
where $v=0.1$ represents the viscosity and $\omega \sim \mathcal{N}(1.2, 1)$. Here, we learn the mapping from the initial condition $u_0(x) := u(x, 0)$ to the solution function $u(x, t)$, i.e.,
\[
\mathcal{G}: u_0(x) \mapsto u(x, t) ,
\]
for $t\in(0,1]$.

The discretized resolution of the spatial domain is 161 and the step size for the temporal domain is 0.01 for the backward Euler method. In this numerical example, we test RDO on different resolutions of the spatial domain. We randomly generate 2000 periodic initial conditions and split the datasets into three subsets with the ratio $0.6, 0.2, 0.2$, respectively, for training, validation, and testing.
For RDO, we set $T_1=3$, $T_2=1$, $d_{\phi_0}= 64$, and $p=512$. The truncated frequency $k$ of FIOs in RDO is set to 8. Furthermore, the branch net and the trunk net in DeepONet are $[41, 512,  512, 512, 512]$ and $[2, 512, 512, 512]$, respectively, where numbers inside the bracket represent the width of each layer. All models are trained on datasets of resolution 41 and tested on the other resolution datasets.

The results of one representative instance are plotted in Figure~\ref{fig:burgers1}. The first row shows the solution with the resolution of the input function being 41 and the corresponding absolute error are plotted in the second row. We can find that the error of RDO is smaller than that of DeepONet. In the remaining rows, the solutions and errors for resolution 81 and 161 are shown, respectively. Moreover, RDO still maintains a low absolute error level.
\begin{figure}[H]
    \begin{center}
            \includegraphics[width=0.75\linewidth]{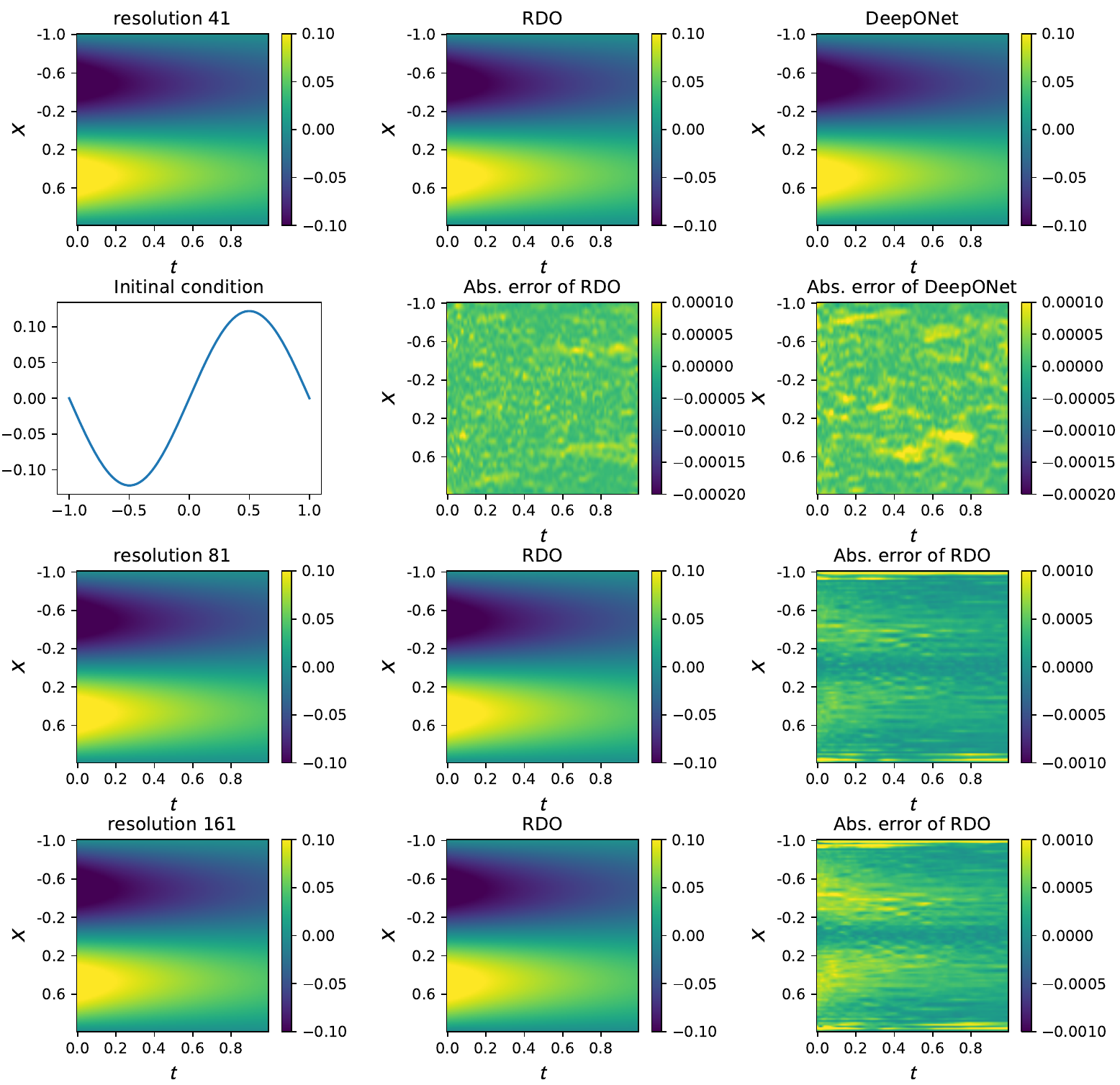}
    \end{center}
    \vspace{-.2cm}
    \caption{One representative example of the Burgers' equation. For resolution 41, the RL2E of DeepONet is $0.6\%$ and while $0.05\%$ for RDO. When the resolution increases to 81 and 161, the RL2E for RDO becomes $0.61\%$ and $0.81\%$, respectively.} \label{fig:burgers1}
\end{figure}

The error rates are summarized in Table~\ref{tab:BurgersProblem}. Since the input dimension of the branch net is fixed, DeepONet is unable to test on other resolutions datasets. In addition, RDO outperforms DeepONet for the test resolution 41.
\begin{table}[H]
    \centering
    \caption{The average relative $L_2$ error on the test dataset for Burgers' equation.}
    \label{tab:BurgersProblem}
    \begin{tabular}{ccc}
        Test resolution & RDO & DeepONet \\
         \hline
          41  &   $0.05\%
         $  & $0.07\%$ \\ 
          81  &   $1.08\%
         $  & - \\ 
          161  &   $1.50\%
         $  & - \\ 
    \end{tabular} 
\end{table}

\section{Conclusion}\label{sec:Conclusion}

In this paper, we propose an extension of the DeepONet, i.e., the resolution-invariant deep operator (RDO), together with the corresponding universal approximation theorem. RDO exhibits the resolution invariant property which implies that it can be trained using low-resolution input and predict for high-resolution without the need of network retraining in contrast with DeepONet. Compared with FNO, our RDO framework can handle problems where the input domain and the output domain are different. Numerical experiments demonstrate that RDO can solve the irregular domain PDE problems and extend to time-dependent problems easily. Compared with existing alternatives in literature,  RDO achieves the best accuracy and has a more flexible framework. 


\bibliographystyle{elsarticle-num-names} 
\bibliography{ref}
\end{document}